\documentclass[12pt, dvipsnames, final, oneside]{amsart}

\usepackage[hypertexnames=false, colorlinks, linkcolor=blue, allcolors=blue, final=true]{hyperref}
\hypersetup{nesting=true,debug=true,naturalnames=true}\usepackage[]{geometry}
\usepackage{ifdraft}
\usepackage{graphicx}
\usepackage{enumitem}
\usepackage{thm-restate}
\usepackage{cleveref}
\usepackage{overpic}
\usepackage{booktabs}
\usepackage{bm}
\usepackage{tikz-cd}
\usepackage{amssymb}

\usepackage[final]{listings}
\usepackage{pythonhighlight}

\usepackage{longtable}
\usepackage{multirow}
\usepackage{multicol}

\usepackage[backend=bibtex,style=alphabetic, doi=false]{biblatex}
\usepackage{hyperref}
\usepackage{xurl}
\hypersetup{breaklinks=true}

\usepackage[colorinlistoftodos]{todonotes}

\newcommand{\benw}[2][]{\ifdraft{\todo[linecolor=Green,backgroundcolor=Green!25,bordercolor=Green,#1]{#2---Ben W.}}{}}
\newcommand{\kgnb}[2][]{\ifdraft{\todo[linecolor=Red,backgroundcolor=Red!25,bordercolor=Red,#1]{#2---Keegan B.}}{}}

\newtheorem{theorem}{Theorem}[section]
\newtheorem{lemma}[theorem]{Lemma}
\newtheorem{conjecture}[theorem]{Conjecture}

\newtheorem{definition}[theorem]{Definition}
\newtheorem{proposition}[theorem]{Proposition}
\newtheorem{corollary}[theorem]{Corollary}

\theoremstyle{definition}
\newtheorem{example}[theorem]{Example}
\newtheorem{remark}[theorem]{Remark}

\newcommand{\N}{\mathbf{N}}
\newcommand{\Q}{\mathbf{Q}}
\newcommand{\Z}{\mathbf{Z}}

\newcommand{\C}{\mathbf{C}}
\newcommand{\F}{\mathbf{F}}
\newcommand{\primeIdealWasl}{\mathfrak{p}}

\DeclareMathOperator{\Gal}{Gal}

\begin{document}
	
	\title{Obstructions to free periodicity and symmetric L-space knots} 
	
	\author{Keegan Boyle}  
	\address{Department of Mathematical Sciences, New Mexico State University, USA} 
	\email{kboyle@nmsu.edu}
	
	\author{Nicholas Rouse}  
	\address{Department of Mathematics, University of British Columbia, Canada} 
	\email{rouse@math.ubc.ca}

        \author{Ben Williams}  
	\address{Department of Mathematics, University of British Columbia, Canada} 
	\email{tbjw@math.ubc.ca}
	
	\setcounter{section}{0}

	\begin{abstract}
		We investigate a polynomial factorization problem that naturally arises from Hartley's factorization condition on the Alexander polynomial of freely periodic knots. We give a number-theoretic interpretation of this factorization condition, which allows for efficient computation. As an application, we prove that any polynomial which is not a product of cyclotomic polynomials can be the Alexander polynomial of a freely $p$-periodic knot for only finitely many $p$. As a demonstration of the computational efficiency of these methods, we also show that the Alexander polynomial of any freely-periodic L-space knot with genus at most 16 must be a product of cyclotomic polynomials. We conjecture that any periodic or freely periodic L-space knot must be an iterated torus knot.
	\end{abstract}
	\maketitle
	\tableofcontents
	
	\section{Introduction}
	The Alexander polynomial of a knot is good at obstructing some types of symmetry. Perhaps most famously, Murasugi proved that the Alexander polynomial of a periodic knot, that is a knot invariant under a finite order rotation around some axis, satisfies a factorization condition \cite{MR292060}. Murasugi's theorem has received significant attention in the literature; see \cite{MR639581} for the version we state here.
	\begin{theorem} \label{thm:Murasugi}
		Let $K$ be an $n$-periodic knot with quotient knot $\overline{K}$, and let $A$ be the axis of symmetry. Then
		\[
		\Delta_K(t) = \Delta_{\overline{K}}(t) \prod_{i = 1}^{n-1}\Delta_{A \cup \overline{K}}(\zeta_n^i,t),
		\]
		where $\zeta_n$ is a primitive $n^{th}$ root of unity.
	\end{theorem}
  A standard use of this theorem is to show that a particular knot is not $n$-periodic for any $n$. 
  
In addition to rotation around an axis, a finite cyclic group $\Z/n\Z$ can also act on $S^3$ freely, and in this case invariant knots are called \emph{freely $n$-periodic}. In \cite{MR608857}, Hartley provides a factorization condition on the Alexander polynomials of freely periodic knots, analogous to the Murasugi condition. We reproduce Hartley's condition here.
	\begin{theorem} \label{thm:Hartley}
		Let $K$ be a freely $n$-periodic knot with Alexander polynomial $\Delta_K(t)$. Then
		\[
		\Delta_K(t^n) = \prod_{i=0}^{n-1}g(\zeta_n^i t),
		\]
		where $\zeta_n$ is a primitive $n^{th}$ root of unity, and $g(t)$ is an Alexander polynomial of a knot.
	\end{theorem}
	Unlike Murasugi's condition, where all values of $n$ may be checked in one computation, Hartley's condition appears to require checking for infinitely many integers $n$. The polynomial to be factored depends on $n$. For certain knots, hyperbolic geometry imposes bounds on the order of any free periodicity, but if we work with only the Alexander polynomial, we do not have recourse to such a luxury as hyperbolic geometry. This paper is written to rule out, by finite computation, free periodicities on various knots using only the information of their Alexander polynomials.

        \par
	We study polynomials satisfying a slightly generalized version of Hartley's condition, by allowing a sign ambiguity.
        \begin{definition} \label{def:nHartley}
          Let $\Delta(t) \in \Z[t]$ be a polynomial and $n$ a positive integer. We say $\Delta(t)$ is \textbf{$n$-Hartley} or \textbf{has the $n$-Hartley condition} if there exists a polynomial $g(t) \in \Z[t]$ for which
          \[ \Delta(t) = \pm \prod_{i=0}^{n-1}g(\zeta_n^i t), \]
          where $\zeta_n$ is a primitive $n$th\benw{changed ``$n$-th'' etc.~to ``$n$th'' throughout to match Wikipedia style (for instance) and to prevent line breaks between the ``$n$'' and the ``th'', which looks stupid.} root of unity.
        \end{definition}
        Note that when $\Delta(t^n)$ is irreducible, $\Delta(t)$ is not $n$-Hartley. Previous work, for example \cite{MR2852954} and \cite[\S 4.1]{SchinzelPolynomials}, has studied whether $\Delta(t^n)$ is irreducible. We do not know of any results that apply to reducible, monic polynomials---a class of particular interest for our applications to symmetric L-space knots as in Conjecture \ref{conj:iteratedtorus} below. \par

        \benw[inline]{Added this remark to cover a bunch of stuff.}
        \begin{remark} \label{rem:backgroundOnSymmetryTypes}
          Periodic and freely-periodic knots are the two extreme examples of a spectrum of symmetries of knots. We refer to \cite[Sections 12, 13]{boyle2023classification} for more detail, where the entire spectrum is contained within the types $\mathrm{Per-}(a)$ and $\mathrm{FPer-}(a,b)$, the latter consisting of freely $n$-periodic knots only when $a$ and $b$ are relatively prime to $n$.\benw{be prepared to change}

          The symmetries in question consist of an action on $S^3$ by an orientation-preserving diffeomorphism $\phi : S^3 \to S^3$ of finite order that satisfies $\phi(K) = K$ and preserves the orientation of the knot---this last condition is automatically satisfied if the order is not $2$. If $\phi$ fixes a point of $S^3$, then it is conjugate to a finite-order rotation, and it is a periodic symmetry of the knot. At the other extreme, if no power $\phi^k$ different from the identity fixes any point of $S^3$, then the symmetry is a free periodicity.
          
          In between, there are the cases where $\phi$ does not fix any point of $S^3$, but some non-identity power $\phi^k$ acts as a periodicity. These may be called \emph{semi-periodic} symmetries, as in \cite{Paoluzzi2018}. Necessarily, a semi-periodic knot is also periodic, albeit of some smaller order of symmetry. Therefore, if one rules out periodic symmetries of all orders for some knot, one rules out all semi-periodic symmetries at the same time. This is of particular relevance to Conjecture \ref{conj:iteratedtorus} below, which also would imply that no L-space knot is semi-periodic, except if it is an iterated torus knot\benw{I thought L-space knots are also expected not to be periodic at all.}.

          We do not know of Murasugi- or Hartley-type conditions on the Alexander polynomial of a semi-periodic knot in general, but in the particular case of a semi-periodicity that is assembled from two periodicities of disjoint axes of rotation (a \emph{biperiodicity}) such a condition is given in \cite[Ch.~3]{Guilbault2021}.
        \end{remark}

\subsection{Results}
Our first achievement is to compute exactly the infinite set of integers $n$ for which the powers of cyclotomic polynomials are $n$-Hartley.
\begin{proposition} \label{prop:cycloHartley}
	Let $f(t) = \Phi_m(t)^k$ be the $k$th power of the $m$th cyclotomic polynomial and let $n \in \N$. Let $b$ be the largest divisor of $n$ for which $\gcd(b,m) = 1$. Then $f(t)$ is $n$-Hartley if and only if $n/b$ divides $k$.
\end{proposition}

In particular, there are infinitely many integers $n$ for which any given power of a cyclotomic polynomial is $n$-Hartley. In contrast we will prove that there are only finitely many such integers for polynomials that are not products of cyclotomic polynomials.

\begin{theorem} \label{thm:main}
	Let $\Delta(t)$ be a polynomial that has a factor that is neither a cyclotomic polynomial nor a monomial. The set of natural numbers $n$ for which $\Delta(t)$ is $n$-Hartley is finite.
      \end{theorem}
This theorem will follow immediately from the results in Section \ref{sec:nt}, and in fact we give an algorithm to compute this finite set in Section \ref{sec:algorithm}. In Section \ref{sec:computations} we apply this algorithm to give computational evidence for some conjectures. 

\begin{corollary}
	Let $\Delta(t)$ be a polynomial that has a factor that is neither a cyclotomic polynomial nor a monomial. There are only finitely many natural numbers $n$ for which there is a freely $n$-periodic knot with Alexander polynomial equal to $\Delta(t)$.
\end{corollary}

Our technique is to define an integer associated to a polynomial $\Delta(t)$ called $D(\Delta(t))$ (Definition \ref{defn:schinzelDegree}) which is defined in terms of the irreducible factors of $\Delta(t)$ and their multiplicities. In Section \ref{sec:nt}, we study how Hartley's condition behaves under factorization. We pay particular attention to powers of irreducible polynomials. Theorem \ref{thm:main} is a consequence of the following more technical result.

\par
\begin{theorem} \label{thm:oneBadFactor}
Suppose that $\Delta(t)$ is the Alexander polynomial of a knot $K$ and that $f_1(t)$ is an irreducible factor of $\Delta(t)$ appearing with multiplicity $s$. If $D(f_1(t)^s) = k$, then $\Delta(t)$ is not $n$-Hartley for any $n$ not dividing $k$. In particular, $K$ can admit free periodicities of orders dividing $k$ only.\benw{changed wording, ideally not the sense.}
\end{theorem}
\begin{example}
Consider the knot \texttt{K12a23} (the alphanumeric code refers to the list \cite{KnotInfo}), which has Alexander polynomial
\[
\Delta(t) = (t - 2) (2t - 1)  (t^4 - 5t^3 + 9t^2 - 5t + 1).
\]
We can easily calculate $D(t-2) = 1$ (see Definition \ref{defn:schinzelDegree}) so that $t-2$ is not $n$-Hartley for any integer $n \ge 2$ by Remark \ref{rmk:iff}. By Theorem \ref{thm:oneBadFactor}, $\Delta(t)$ is not $n$-Hartley for any $n$, and so \texttt{K12a23} is not freely periodic of any order.
\end{example}
Our results produce a practical algorithm for determining those $n$ for which a given polynomial is $n$-Hartley. In fact, we resolve Conjecture \ref{conj:polynomial} below for all knots of genus less than $13$ unconditionally and less than $19$ assuming the L-space conjecture (see below). Our computation involves calculating $D(\Delta(t))$ for over $130,000$ polynomials.
	
\subsection{Conjectures on symmetric L-space knots}

We provide computational evidence for some conjectures about symmetric L-space knots. L-space knots are those knots that have a surgery that is an L-space, a 3-manifold with a particularly simple form of Heegaard Floer homology \cite[Definition 1.1]{MR2168576}. \par
One motivation for the conjectures we present here is that it is difficult to construct polynomials that simultaneously satisfy the restrictions required of the Alexander polynomial of an L-space knot (see for example \cite{MR2168576}, and \cite{MR3782416}), and also satisfy either the factorization condition from Theorem \ref{thm:Murasugi} or from Theorem \ref{thm:Hartley}.
\begin{conjecture} \label{conj:iteratedtorus}
  Every periodic or freely periodic L-space knot is an iterated torus knot.
\end{conjecture}
      
\begin{conjecture} \label{conj:polynomial}
	Let $K$ be a periodic or freely periodic L-space knot. Then $\Delta_K(t)$ is a product of cyclotomic polynomials. 
\end{conjecture}
\begin{conjecture} \label{conj:algebra}
Let $f(t)$ be a polynomial whose non-zero coefficients alternate between $1$ and $-1$ and suppose that $f(t)$ is not a product of cyclotomic polynomials. Then $f(t)$ is not $n$-Hartley for any $n > 1$.
\end{conjecture}
In Section \ref{sec:computations} we verify Conjecture \ref{conj:algebra} for polynomials of degree less than $37$.
\begin{conjecture} \label{conj:quotient}
	Let $K$ be a periodic L-space knot with quotient knot $\overline{K}$. Then $\overline{K}$ is an L-space knot.
\end{conjecture}

\begin{remark}
	Conjecture \ref{conj:iteratedtorus} implies Conjectures \ref{conj:polynomial} and \ref{conj:quotient}, since the Alexander polynomial of an iterated torus knot is a product of cyclotomic polynomials (this follows from the cabling formula for the Alexander polynomial), and the quotient of a periodic iterated torus knot is again an interated torus knot and hence is an L-space knot by \cite{MR3546454}.
\end{remark}

\begin{remark}
	There is also a conjecture of Li and Ni \cite[Conjecture 1.3]{MR3485330} that L-space knots with Alexander polynomials which are products of cyclotomic polynomials are iterated torus knots. This would imply that Conjecture \ref{conj:polynomial} is equivalent to Conjecture \ref{conj:iteratedtorus}.
\end{remark}

\begin{remark}
Conjecture \ref{conj:algebra} implies Conjecture \ref{conj:polynomial}, since Alexander polynomials of L-space knots satisfy the given conditions; see \cite[Corollary 1.3]{MR2168576}.
\end{remark}

\begin{remark}
In Conjecture \ref{conj:algebra}, each condition is necessary. Indeed, consider the polynomials
\begin{itemize}
  \item $t^{10} - t^9 - t^7 + t^6 - t^5 + t^4 - t^3 - t + 1$, which has non-alternating coefficients and is 2-Hartley,
  \item $t^2 - 3t + 1$, which has a coefficient not in $\{-1,0,1\}$ and is 2-Hartley, and
  \item $t^2 - t + 1$, which is cyclotomic and 5-Hartley.
\end{itemize}
Furthermore, these polynomials are all symmetric and have $f(1) = \pm 1$ so that even within the class of Alexander polynomials each condition is necessary.
\end{remark}

\begin{remark}
	It was once conjectured that L-space knots must be strongly invertible; see \cite[Section 6]{MR3649241}. However, Baker and Luecke have produced examples of L-space knots with no symmetries \cite[Theorem 1.2]{MR4194294}. Conjecture \ref{conj:iteratedtorus} is then a partial converse, claiming that L-space knots cannot have more than one strong inversion.
\end{remark}

The ``L-space conjecture'' of  \cite{MR3072799} claims that an irreducible rational homology $3$-sphere has a left-orderable fundamental group if and only if it is not an L-space. We prove in Corollary \ref{cor:Lspace} that the L-space conjecture implies Conjecture \ref{conj:quotient}, and we observe that Conjecture \ref{conj:iteratedtorus} is true for the $P(-2,3,2q+1)$ pretzel knots by \cite{MR891592}.

\subsection{Notation and conventions}
\label{sec:notation-conventions}

We write $\bar \Q \subset \C$ for the field of algebraic numbers, i.e., the algebraic closure of $\Q$ in $\C$. If $n$ is a positive integer, then  $\zeta_n =\exp(2 \pi i/n)$ denotes a primitive $n$th root of $1$, and $\Phi_n(t) \in \Z[t]$ denotes its minimal polynomial, the $n$th cyclotomic polynomial.

If $L/K$ is an extension of fields with Galois group $G$, and if $\sigma \in G$ and $f(t) \in L[t]$, then we write $\sigma f$ to denote the polynomial in $L[t]$ whose coefficients are obtained from those of $f$ by applying $\sigma$. Note that $\sigma(f(a)) = \sigma f (\sigma a) \neq \sigma f(a)$ in general.

A nonzero polynomial $f(t) \in \Z[t]$ is said to be \emph{primitive} if the greatest common divisor of the coefficients of $f$ is $1$. A polynomial $f(t) \in \Z[t]$ is said to be \emph{irreducible} if it is irreducible as an element of $\Q[t]$, so that for a nonzero $f(t) \in \Z[t]$,  ``irreducible and primitive'' is synonymous with ``irreducible in $\Z[t]$''.

\subsection{Acknowledgments} We would like to thank Liam Watson for helpful conversations. The second author is supported by the Pacific Institute for Mathematical Sciences (PIMS) through the grant ``Novel Techniques in Low Dimensions''.

The third author acknowledges the support of the Natural Sciences and Engineering Research Council of Canada (NSERC), RGPIN-2021-02603 and of PIMS through the grants ``Novel Techniques in Low Dimensions'' and CRG41.
\section{Number-theoretic interpretation of Hartley's condition} \label{sec:nt}

We define a slightly more general condition that agrees with Hartley's factorization condition when the polynomial in question is an Alexander polynomial.
\begin{definition} \label{defn:Hartley}
  Let $\Delta(t)$ be a polynomial with integer coefficients and $n \in \Z_{\geq 2}$. We say that $\Delta(t)$ is \textbf{n-Hartley} if there exists $g(t) \in \Z[t]$ such that 
  \begin{equation}
    \label{eq:3}
      \Delta(t^n) = \pm \prod_{i=0}^{n-1}g(\zeta_n^i t).
  \end{equation}
\end{definition}
We prove now that every Alexander polynomial that is $n$-Hartley and is normalized to have positive constant term satisfies Hartley's factorization condition as given in Theorem \ref{thm:Hartley}.
\begin{proposition} \label{prop:nHartleyImpliesHartley}
  Let $\Delta(t)$ be $n$-Hartley with positive constant term. Then $\Delta(t)$ satisfies Hartley's factorization condition \ref{thm:Hartley}.
\end{proposition}

\begin{proof}
  Suppose that $\Delta(t^n)$ factors as
  \[
    \Delta(t^n) = (-1)^j \prod_{i=0}^{n-1}g(\zeta_n^i t)
  \]
  for $j=0$ or $1$. First observe that if $j=1$ and $n$ is odd, then we can replace $g(t)$ with $-g(t)$ to obtain $j=0$, so suppose $n$ is even. Then $\Delta(0) = (-1)^j g(0)^n$, so $j=0$.
\end{proof}
The two conditions are not the same, however. The reader may easily verify that $t-1$ is $2$-Hartley in the sense of this paper, but does not satisfy Hartley's original condition.\benw{I added this comment in lieu of a confusing aside before Example 2.7 below.}

\benw[inline]{Added, because we use it later, and it is easy to put here.}
The next lemma is easy to prove, and is useful to rule out the $n$-Hartley condition from a great many polynomials.
\begin{lemma} \label{lem:leadingConst}
  Let $\Delta(t)$ be $n$-Hartley, and let $a,b$ denote the leading and constant coefficients of $\Delta(t)$. Then $|a|$ and $|b|$ are $n$-th powers in $\Z$.
\end{lemma}
\begin{proof}
  Suppose a factorization as in \eqref{eq:3} exists, and that the degree of $g(t)$ is $r$. Let $c$ denote the leading coefficient of $g(t) \in \Z[t]$. Comparing leading coefficients, we see
  \[ a = \pm c^n \prod_{i=0}^{n-1} \zeta_n^{ir}. \]
  Taking moduli on each side gives $|a| = |c|^n$.

  The argument for the constant coefficient is similar.
\end{proof}
This lemma gives a quick proof of Theorem \ref{thm:main} except when the leading and constant coefficients of $\Delta(t)$ lie in $\{-1,1\}$. The rest of this section is devoted to proving the remaining cases. Although we develop the theory in some generality, the reader might bear in mind that it will be applied only to polynomials whose leading and constant coefficients lie in $\{-1,1\}$.

\medskip

The following lemma and its corollary will be used to establish the Hartley condition using field-theoretic methods.

\begin{lemma}\label{lem:primHartley}
  Let $n$ be a natural number. Suppose $\Delta(t), g(t) \in \Z[t]$ are polynomials that satisfy
  \[ \Delta(t^n) = \pm \prod_{j=0}^{n-1} g(\zeta_n^j t). \]
  The polynomial $\Delta(t)$ is primitive if and only if $g(t)$ is primitive.
\end{lemma}
\begin{proof}
  In one direction, this is straightforward. Suppose $\Delta(t)$ is primitive. The polynomial $g(t)$ is a factor of $\Delta(t)$ in $\Z[t]$, which implies that $g(t)$ is primitive.

  We prove the other direction. We suppose $g(t)$ is primitive. The nonzero coefficients of $\Delta(t^n)$, as a set, coincide with the nonzero coefficients of $\Delta(t)$, so it suffices to prove that $\Delta(t^n)$ is primitive. We prove that there is no prime number $p$ that divides all the coefficients of $\Delta(t^n)$.

  Let $p$ be an arbitrary prime number and fix an algebraic closure $\bar \F_p$ of the field with $p$ elements. There is a unique homomorphism $h_p \colon \Z[t] \to \bar \F_p[t]$ of rings with the property that $t \mapsto t$.
  
  There exists an extension of $h_p$ to
  \[ \bar h_p : \Z[\zeta_n, t] \to \bar \F_p[t] \]
  given by defining $\bar h_p(\zeta_n)$ to be some root of $h_p(\Phi_n(t)) \in \bar \F_p[t]$. We see that $\bar h_p(\zeta_n) \neq 0$, since in particular $h_p(\zeta_n)$ is a root of $h_p(\Phi_n(t))$, and therefore of $t^n-1$.

  For each $j \in \{1, \dots, n-1\}$, the coefficients of the polynomial $g(\zeta_n^j t)$ consist of some power of $\zeta_n$ times a coefficient of $g(t)$. Therefore the coefficients of $\bar h_p(g(\zeta_n^j t))$ consist of some power of $\bar h_p(\zeta_n)$ times a coefficient of $\bar h_p(g(t))$, and the coefficients of $\bar h_p(g(t))$ are not all $0$ since $g(t)$ is primitive. Notably, $\bar h_p(g(\zeta_n^j t)) \neq 0$.

  Applying $\bar h_p$ to the factorization gives us
  \[ h_p(\Delta(t^n)) = \bar h_p (\Delta(t^n)) = \pm \prod_{j=0}^{n-1} \bar h_p(g(\zeta_n^j t)) \]
  where the left equality holds because $\Delta(t^n)$ is an integer polynomial. The product on the right is a product of nonzero terms in $\bar \F_p[t]$, which is an integral domain. We conclude that $h_p(\Delta(t^n)) \neq 0$.

  Since $p$ is an arbitrary prime number, we deduce that $\Delta(t^n)$, and therefore $\Delta(t)$, is a primitive polynomial.
\end{proof}

\begin{corollary}\label{cor:QHartleyZHartley}
  Let $n$ be a positive integer. Suppose $\Delta(t), g(t) \in \Z[t]$ are primitive polynomials, and that a factorization
    \begin{equation}
      \label{eq:1}
      \Delta(t^n) = u \prod_{j=0}^{n-1} g(\zeta_n^j t)
    \end{equation}
  exists where $u\in \bar \Q^\times$. Then $u \in \{-1,1\}$, and in particular $\Delta(t)$ is $n$-Hartley.
\end{corollary}
\begin{proof}
  As observed in the proof of Lemma \ref{lem:primHartley}, the polynomial $\Delta(t^n)$ is also primitive.

  The elements $\sigma$ of the Galois group $\Gal(\Q(\zeta_n)/\Q)$ act by permuting the powers $\zeta_n^j$. In particular, any such $\sigma$ permutes the polynomials $g(\zeta_n^j t)$, and therefore fixes the polynomial
  \[ \prod_{i=0}^{n-1} g(\zeta_n^j t),\]
  which therefore lies in $\Q[t]$. Its coefficients are algebraic integers, so in fact it lies in $\Z[t]$. It follows that $u$ must lie in $\Q^\times$. We may write
  \[ a \Delta(t^n) = b \prod_{j=0}^{n-1}g(\zeta_n^j t) \]
  for some relatively prime integers $a,b$. Since $\Delta(t^n)$ is primitive, we see that $b=\pm 1$. Using Lemma \ref{lem:primHartley}, the primitivity of $g(t)$ implies that $a \Delta(t^n)$ is primitive, so that $a=\pm 1$.
\end{proof}

\kgnb[inline]{There used to be a lemma that if $\Delta(t)$ is irreducible, so is $g(t)$, but it wasn't cited anywhere, so I've removed it. Anyway, if it's useful at all we can add it back in.}
\benw[inline]{I think we should add it back. It combines with Corollary \ref{cor:QHartleyZHartley} in the following way: Suppose $\Delta(t)$ is irreducible and primitive in $\Z[t]$. If $\Delta(t)$ has the following factorization
  \[ \Delta(t) = u \prod_{i=1}^n g(\zeta_n^it), \quad u \in \bar \Q^\times \]
  over $\bar \Q$, then $\Delta$ is $n$-Hartley.
This is actually an if \& only if statement, the other direction being easier. Anyway, it does allow us to reduce the testing of the Hartley property entirely to $\Q$-calculations.
}
\kgnb[inline]{I've added it back in (below). To do: modify it to address Ben's comment above (and maybe merge the statement with the previous lemma?)}
\benw[inline]{It turns out not to be necessary after all. I changed ``lemma'' to ``proposition'' because we don't use it, and I'd be open to omitting it again.}
\kgnb[inline]{It's commented below.}


\begin{proposition} \label{prop:coprimeHartley}
  Let $f(t), h(t) \in \Z[t]$ be coprime and primitive. Then $f(t)h(t)$ is $n$-Hartley if and only if $f(t)$ and $h(t)$ are $n$-Hartley. 
\end{proposition}
\begin{proof}
  The reverse direction is straightforward; indeed, if $f(t)$ and $h(t)$ are each $n$-Hartley, then we have
  \[
    f(t^n)h(t^n) = \prod_{i=0}^{n-1}F(\zeta_n^i t) \prod_{i=0}^{n-1}H(\zeta_n^i t) = \prod_{i=0}^{n-1}F(\zeta_n^i t)H(\zeta_n^i t),
  \]
  for some polynomials $F$ and $H$.
  
  Conversely, suppose that $f(t)h(t)$ is $n$-Hartley so that there is $g(t) \in \Z[t]$ with
  \begin{equation} \label{eq:hartleyfactor}
    f(t^n)h(t^n) = \prod_{i=0}^{n-1}g(\zeta_n^i t).
  \end{equation}
  We will partition the roots of the right hand side of Equation \ref{eq:hartleyfactor} into roots of $f(t^n)$ and $h(t^n)$. To begin, factor $g(t)$ into irreducible factors $g(t) = g_1(t)g_2(t) \dots$. Since $f(t)$ and $h(t)$ are coprime, $f(t^n)$ and $h(t^n)$ have no roots in common, so each $g_j(t)$ divides exactly one of $f(t^n)$ or $h(t^n)$. Futhermore, note that $g_j(\zeta_n^i t)$ divides $f(t^n)$ if and only if $g_j(t)$ does. Now let $F(t)$ be the product of the factors of $g(t)$ dividing $f(t^n)$ and let $H(t)$ be the product of the factors of $g(t)$ dividing $h(t^n)$. Then $g(t) = F(t)H(t)$, and in fact
  \[
    f(t^n) = c_1 \prod_{i=0}^{n-1}F(\zeta_n^i t), \mbox{ and } h(t^n) = c_2 \prod_{i=0}^{n-1}H(\zeta_n^i t)
  \]
  for some $c_1,c_2 \in \Z$. Since $f(t)$ and $h(t)$ are primitive, we have $c_1,c_2 = \pm1$.
\end{proof}

 Proposition \ref{prop:coprimeHartley} motivates our mild generalization of Hartley's original condition to allow for a minus sign. It would not hold without that generalization, as can be deduced from the following example.
\begin{example} \label{coexmp:fake2Hartley}
  Let $K$ be the knot \texttt{K14n26330}. Its Alexander polynomial is
  \[
    \Delta(t) = 4t^6 - 17t^5 + 38t^4 - 51t^3 + 38t^2 - 17t + 4,
  \] so that
  \[
    \Delta(t^2) = (2t^6 - 3t^5 - 2t^4 + 7t^3 - 2t^2 - 3t + 2)(2t^6 + 3t^5 - 2t^4 - 7t^3 - 2t^2 + 3t + 2),
  \]
  so $\Delta(t)$ is $2$-Hartley. However $\Delta(t)$ factors into irreducibles as
  \[
    \Delta(t) = (t^3 - 3t^2 + 5t - 4)(4t^3 - 5t^2 + 3t - 1),
  \]
  and neither irreducible factor satisfies Hartley's original condition with the positive sign and $p=2$.

  If we write $f_1(t)$ and $f_2(t)$ for the irreducible factors of $\Delta(t)$, then $f_1(t^2)$ factors as $-h_1(t)h_1(-t)$ for $h_1(t) = t^3 - t^2 - t + 2$, and $f_2(t^2)$ factors as $-h_2(t)h_2(-t)$ for $h_2(t) = 2t^3 - t^2 - t + 1$. Then $\Delta(t^2) = f_1(t^2)f_2(t^2) = h_1(t)h_2(t)h_1(-t)h_2(-t)$, and $\Delta(t)$ does satisfy Hartley's condition with $p=2$.
\end{example}
        
Most of our results hinge on the following characterization of $n$-Hartley polynomials. 
\begin{proposition} \label{prop:pthPowerIffpHartleyBen}
  Let $s$, $n$ be positive integers. Suppose $\Delta(t) = h(t)^s \in \Z[t]$ is an $s$th power of an irreducible, primitive, nonconstant polynomial $h(t)$. Let $\alpha$ be a root of $\Delta(t)$. The polynomial $\Delta(t)$ is $n$-Hartley if and only if there exists a polynomial $q(t) \in \Q(\alpha)[t]$ having the following properties:
  \begin{enumerate}
  \item $\deg(q(t)) = s$;
  \item every root of $q(t)$ is an $n$th root of $\alpha$.
  \end{enumerate}
\end{proposition}
\begin{proof}
  The proposition is trivially true if $h(t) = \pm t$. Therefore we can assume for the rest of the proof that $h(t)$ does not have $0$ as a root. The rest of the proof is in two parts.
  \smallskip  

  Suppose a polynomial $q(t)$ exists satisfying the given conditions. We will construct $g(t)$ as a product of Galois conjugates of $q(t)$.

  Let $F$ be a splitting field for $h(t)$ over $\Q$ and write $G = \Gal(F/\Q)$. This group acts transitively on the roots of the irreducible polynomial $h(t)$. Let $H \subseteq G$ denote the stabilizer of $\alpha$ under this action.

  Choose a set $\gamma_1, \dots, \gamma_m$ of coset representatives for $G/H$ so that $\gamma_1 \alpha, \dots , \gamma_m \alpha$ constitute the roots of $h(t)$. Recall that we write $\gamma_iq(t)$ to denote the polynomial whose coefficients are obtained by applying $\gamma_i$ to the coefficients of $q$. Observe that $q(t)$ is invariant under the action of $H$, since its coefficients lie in $\Q(\alpha)$.
  
  Now let us define the polynomial
  \[ g(t) = \prod_{i=1}^m \gamma_iq(t) \]
  which has degree $ms$. Its roots, taken with multiplicity, are $s$ solutions to each of the $m$ different equations $t^n = \gamma_i \alpha$ for $i \in \{1, \dots, m\}$.

  We claim $g(t)$ is defined over $\Q$. It is certainly defined over $F$. If $\sigma \in G$ is an element, then for any $i \in \{1,\dots, m\}$ there exists some $j \in \{1, \dots, m\}$ for which the cosets $\sigma \gamma_i H$ and $\gamma_jH $ coincide. Then
  \[ \sigma(\gamma_iq) (t) = (\sigma\gamma_i)q(t) = \gamma_jq(t). \]
  Therefore $g(t)$ is invariant under the action of $G$, and consequently its coefficients lie in $\Q$, establishing our claim.

  We now clear denominators in $g(t)$ to yield a primitive polynomial in $\Z[t]$ having the same roots, which we continue to denote $g(t)$ to keep the notation simple.
  
  We define the polynomial
  \[ d(t)= \prod_{j=0}^{n-1} g(\zeta_n^j t), \]
  which has degree $mns$. As we will see below, this differs from $\Delta(t)$ only by a unit, allowing us to apply \ref{cor:QHartleyZHartley}.

  For the moment, we fix some $i \in \{1, \dots, m\}$. The polynomial $g(t)$ has exactly $s$ roots (counted with multiplicity) that satisfy $t^n = \gamma_i \alpha$, as observed above. Let $\beta$ be an $n$th root of $\gamma_i$, so that these $s$ roots are
  \[ \zeta_n^{a_1} \beta, \zeta_n^{a_2} \beta, \dots , \zeta_n^{a_s} \beta \] where
  $a_1, \dots, a_s$ is a list of elements of $\Z/n\Z$. The polynomial $d(t)$ has exactly $ns$ roots (counted with multiplicity) that satisfy $t^n = \gamma_i \alpha$. Listed, they are
  \[  \zeta_n^{a_1-0}\beta, \zeta_n^{a_1-1}\beta, \dots, \zeta_n^{a_1 - n + 1}\beta, \zeta_n^{a_2-0}\beta, \zeta_n^{a_2-1}\beta, \dots, \zeta_n^{a_2-n+1}\beta, \dots, \zeta_n^{a_s-0}\beta, \zeta_n^{a_s-1}\beta, \dots, \zeta_n^{a_s-n+1}\beta. \]
  We observe that each power $\zeta_n^j$ appears $s$ times in this list, no matter what the values of $a_1, \dots, a_s$ are. The list is precisely the roots of $(t^n-\gamma_i \alpha)^s$.

  Since this holds for an arbitrary $i$, we deduce that the roots of $d(t)$, with multiplicity, are exactly the roots of $h(t^n)^s = \Delta(t^n)$. Therefore we deduce
  \[ \Delta(t^n) = u \prod_{i=0}^{n-1} g(\zeta_n^j t) \]
  for some scalar $u \in \bar \Q^\times$. Since $h(t)$ is primitive, $\Delta(t^n)$ is primitive and Corollary \ref{cor:QHartleyZHartley} applies to tell us that $\Delta(t)$ is $n$-Hartley.

  \smallskip
  Now we prove the reverse direction. Suppose $\Delta(t)=h(t)^s$ is $n$-Hartley, and write
  \[ \Delta(t^n) = \pm \prod_{j=0}^{n-1} g(\zeta_n^j t), \]
  where the polynomial $g(t)$ is defined over $\Z$. We wish to produce the polynomial $q(t)$.

  Since $\alpha$ is an $s$-fold root of $\Delta(t)$, there are $n$ different roots of $\Delta(t^n)$ that also satisfy $t^n - \alpha =0$. They may be listed as
  \[ \beta, \zeta_n \beta, \zeta_n^2\beta , \dots \zeta_n^{n-1} \beta \]
  each appearing with multiplicity $s$. With multiplicity, $s$ of these are roots of $g(t)$, another $s$ are roots of $g(\zeta_n t)$ and so on.
 Now construct the greatest common divisor $q(t)$ of $(t^n - \alpha)^s$ and $g(t)$ in the ring $\Q(\alpha)[t]$. This has exactly $s$ roots, all of which are $n$th roots of $\alpha$, and is the polynomial we wanted.  
\end{proof}

Determining whether $q(t)$ exists in Proposition \ref{prop:pthPowerIffpHartleyBen} may be time consuming. The following is a useful shortcut for ruling out the existence of such a $q(t)$.
\begin{corollary}\label{cor:powerRuleFornHartley}
  Let $s,n$ be positive integers. Suppose $\Delta(t) = h(t)^s \in \Z[t]$ is an $s$th power of an irreducible primitive nonconstant polynomial $h(t)$. Let $\alpha$ be a root of $\Delta(t)$. If $\Delta(t)$ is $n$-Hartley, then there exists an element $\theta \in \Q(\alpha)$ such that $\theta^n = \alpha^s$.
\end{corollary}
\begin{proof}
  Assuming $\Delta(t)$ is $n$-Hartley, then a polynomial $q(t) \in \Q(\alpha)[t]$ exists as in Proposition \ref{prop:pthPowerIffpHartleyBen}. Let $\theta$ be $(-1)^s$ times the constant coefficient of $q(t)$. It is the product of $s$ values, all of which are $n$th roots of $\alpha$, so that $\theta^n = \alpha^s$, as required.
\end{proof}

When $s=1$, Proposition \ref{prop:pthPowerIffpHartleyBen} simplifies considerably and the converse of Corollary \ref{cor:powerRuleFornHartley} holds. We give the statement of this case as corollary for ease of reference. It is simply a special case of Proposition \ref{prop:pthPowerIffpHartleyBen}.
\begin{corollary} \label{cor:iff}
  Let $n$ be a positive integer. Suppose $\Delta(t) \in \Z[t]$ is an irreducible, primitive, nonconstant polynomial. Let $\alpha$ be a root of $\Delta(t)$. The polynomial $\Delta(t)$ is $n$-Hartley if and only if $\alpha$ has an $n$th root in $\Q(\alpha)$.
\end{corollary}
 
\begin{remark} \label{rmk:iff}
  In general, the converse of Corollary \ref{cor:powerRuleFornHartley} fails. For example, let $m$ be a positive integer and consider the polynomial $\Delta(t) = (t+m^4)^2$. The unique root $\alpha$ of this polynomial is $-m^4$, and $\Q(\alpha) = \Q$.  Set $n=4$. The field $\Q$  contains an element $\epsilon$ for which $\epsilon^4 = \alpha^2$, viz., $\epsilon=m^2$. But the polynomial $\Delta(t)$ does not satisfy the $4$-Hartley condition. If it did, there would be a degree-$2$ polynomial $q(t)$ defined over $\Q$ whose roots would be drawn from the set of $4$th roots of $\alpha=-m^4$, which is $\{\zeta_8m, \zeta_8^3m, \zeta_8^5m, \zeta_8^7m\}$. Since $\zeta_8$ is defined only after a degree-$4$ extension of $\Q$, this is impossible.
\end{remark}
We also observe that we cannot reduce all calculations to the case when $s =1$ as in Corollary \ref{cor:iff}. Indeed, one might hope that whenever $h(t)^s$ is $(sn)$-Hartley, that $h(t)$ is $n$-Hartley. This is not the case, even for Alexander polynomials.
\begin{example}
Consider the primitive and irreducible polynomial $\Delta(t) = t^8 + 2t^7 - 5t^6 + 5t^4 - 5t^2 + 2t + 1$, which is symmetric and satisfies $\Delta(1) = 1$ so that $\Delta(t)$ is the Alexander polynomial of some knot (see e.g., \cite{Levine}). A computer easily verifies that $\Delta(t^2)$ is irreducible and so $\Delta(t)$ is not $2$-Hartley. However, $\Delta(t)^2$ is $4$-Hartley with factor
	\[
	g(t) = t^{16} - 2t^{15} + 2t^{14} - t^{12} + 2t^{11} - 2t^{10} + t^8 - 2t^6 + 2t^5 - t^4 + 2t^2 - 2t + 1.
	\]
\end{example}
The following proposition will also reduce the necessary computations to determine the complete set of integers $n$ for which a polynomial is $n$-Hartley.
\begin{proposition} \label{prop:primepowers}
  Let $n, m$ be positive integers for which $m \mid n$. If $\Delta(t) \in \Z[t]$ is a primitive, nonconstant $n$-Hartley polynomial, it is $m$-Hartley.
\end{proposition}
\begin{proof}
  By use of Proposition \ref{prop:coprimeHartley}, we can reduce the problem to the case where $\Delta(t) = h(t)^s$ is an $s$th power of an irreducible primitive polynomial.

  Assume $\Delta(t)$ is $n$-Hartley. Let $\alpha$ be a root of $\Delta(t)$.  Proposition~\ref{prop:pthPowerIffpHartleyBen} tells us that there exists a degree-$s$ polynomial $q(t) \in \Q(\alpha)[t]$, every root of which is an $n$th root of $\alpha$. Let $\beta_1, \dots, \beta_s$ be these roots, listed with multiplicity. Without loss of generality, let us scale $q(t)$ so that it is monic.

  Write $d$ for $n/m$, which is a positive integer. Let $q_0(t)$ be the monic polynomial over $\bar \Q$ whose roots are $\beta_1^d, \beta_2^d, \dots, \beta_s^d$. The polynomial $q_0(t)$ has degree $s$. Its roots are $m$th roots of $\alpha$. Its coefficients are symmetric polynomials in $\beta_1, \dots, \beta_s$, and therefore may themselves be expressed as polynomials in the coefficients of $q(t)$. In particular, $q_0(t) \in \Q(\alpha)[t]$. Therefore $\Delta(t)$ has the $m$-Hartley condition, by Proposition \ref{prop:pthPowerIffpHartleyBen}.          
\end{proof}

\begin{proof}[Proof of Proposition \ref{prop:cycloHartley}]
Let us write $n$ as a product of two integers $n=ab$, where the prime factors of $a$ are a subset of those of $m$, and $b$ is relatively prime to $m$. For later use, we choose some positive integer $c$ for which $bc \equiv 1 \pmod m$.

We first claim that the polynomial $\Phi_m(t)^k$ is $n$-Hartley if and only if it is $a$-Hartley.

Certainly if it is $n$-Hartley, it is $a$-Hartley, by Proposition \ref{prop:primepowers}. Conversely, suppose it is $a$-Hartley. Then by Proposition \ref{prop:pthPowerIffpHartleyBen}, there exists a degree-$k$ polynomial $q(t) \in \Q(\zeta_m)[t]$ whose roots are $a$th roots of the root $\zeta_m^c$ of $\Phi(t)$. If $\beta$ satisfies $\beta^a = \zeta_m^c$, then raising both sides to the power of $b$ gives us $\beta^n = \zeta_m^{bc} = \zeta_m$. Therefore the roots of $q(t)$ are also $n$th roots of $\zeta_m$, which is also a root of $\Phi_m(t)^k$. Proposition \ref{prop:pthPowerIffpHartleyBen} now implies that $\Phi_m(t)^k$ is $n$-Hartley, establishing our claim.

\smallskip

We must therefore determine the conditions under which $\Phi_m(t)^k$ is $a$-Hartley. We will use Proposition \ref{prop:pthPowerIffpHartleyBen} again. Consequently, our next task is to determine the minimal polynomials satisfied by $a$th roots of $\zeta_m$ over $\Q(\zeta_m)$.

The $a$th roots of $\zeta_m$ are the numbers
\[ \zeta_{am}, \zeta_{am}^{m+1}, \zeta_{am}^{2m+1}, \dots , \zeta_{am}^{am-m+1}, \]
all of which are primitive $am$th roots of unity, since the exponents $jm+1$ are relatively prime to $m$ and therefore also to $am$, which has the same prime divisors. These roots lie in the field $\Q(\zeta_{am})$. We determine that the degree $[\Q(\zeta_{am}): \Q(\zeta)]$ is $a$ by considering the degrees of the cyclotomic extensions of $\Q$ in the following tower:
\[
  \begin{tikzcd}
    \Q(\zeta_{am}) \arrow[dd, no head, "\phi(am)"'] \arrow[dr, no head, "a"]&  \\ & \Q(\zeta_m) \arrow[dl, no head, "\phi(m)"] \\ \Q. &
  \end{tikzcd}
\]
Here the edges are labelled by the degrees of the field extensions, and $\phi(x)$ is Euler's totient function $\phi(x) = x \prod_{p}\big(1-\frac{1}{p}\big)$ where the product is taken over prime divisors $p$ of $x$. The degree $a=[\Q(\zeta_{am})\colon \Q(\zeta_m)]$ is calculated by observing that $m$ and $am$ have the same prime divisors. In particular, the degree of a minimal polynomial of $\zeta_{am}$ over $\Q(\zeta_m)$ is $a$. Since $\zeta_{am}$ satisfies the degree-$a$ polynomial $t^a - \zeta_m$, we conclude that  this is the minimal polynomial and is, in particular, irreducible over $\Q(\zeta_m)$. Its roots are all the $a$th roots of $\zeta_m$.

According to Proposition \ref{prop:pthPowerIffpHartleyBen}, the polynomial $\Phi_m(t)^k$ is $a$-Hartley if and only if there exists a degree-$k$ polynomial $q(t)$ defined over $\Q(\zeta_m)$ all of whose roots are $a$th roots of $\zeta_m$. Any polynomial $q(t)$ defined over $\Q(\zeta_m)$ all of whose roots are $a$th roots of $\zeta_m$ must be of the form $(t^a-\zeta_m)^j$ for some positive integer $j$. Therefore, such a polynomial can have degree $k$ if and only if $a$ divides $k$. Since $a=n/b$, this proves the result.
\end{proof}

\par

Now we set about producing an algorithm that can determine the integers $n$ for which a given $\Delta(t)$ is $n$-Hartley. \benw{rewritten a bit.}

Inspired by \cite[\S 4.1]{SchinzelPolynomials}, we make the following definition. 
	\begin{definition} \label{defn:schinzelDegree}
		Let $\alpha$ be a nonzero\benw{added ``nonzero''.} algebraic number. Let
		\[
		D(\alpha) = \sup\{D \in \Z_{\ge 0} \mid \alpha = \theta^D \text{ for some } \theta \in \Q(\alpha)\},
              \]
              the supremum being taken with respect to divisibility.
		For a polynomial $f(t) \in \Q[t]$ of at least two terms whose distinct roots $\alpha_i$ appear with multiplicity $m_i$, define
		\[
		D(f(t)) = \gcd_i(D(\alpha_i^{m_i})).
		\]
	\end{definition}
        Every set of nonnegative integers has a supremum with respect to divisibility: if the set is infinite then the supremum is $0$, the maximal nonnegative integer with respect to this ordering.

        In particular,
        \[ D(\alpha) = 0 \quad \text{ if $\alpha$ is a root of unity,} \]
        since (for instance) any such $\alpha$ satisfies $\alpha^D = \alpha$ for infinitely many nonnegative integers $D$.


          
        
        In all cases, $D(\alpha)$ is Galois invariant, in that $D(\alpha) = D(\sigma(\alpha))$ for all $\sigma$ in the Galois group of $\Q(\alpha)$ over $\Q$. 
        
        If $\alpha$ is not a root of unity and $n \mid D(\alpha)$, then $\alpha$ admits an $n$th root in $\Q(\alpha)$. To see this, use the Chinese remainder theorem to reduce to the case where $n=p^r$ is a power of a prime, where it is immediate since the supremum was taken with respect to divisibility. In particular, $\alpha$ admits a $D(\alpha)$th root in $\Q(\alpha)$, i.e., the supremum is actually attained.

       If $\alpha$ is not a root of unity, then $D(\alpha)$ is a positive integer, so that for a polynomial $f(t)$, one has $D(f(t)) = 0$ precisely when $f(t)$ is a product of cyclotomic polynomials.\benw{rewrote the preceding text a bit.}
       If $\alpha$ is an algebraic number, then we say $\alpha$ is an \emph{algebraic unit} if $\alpha$ is in $\mathcal{O}_{\Q(\alpha)}^\times$, the group of units of the ring of integers in $\Q(\alpha)$.\benw{added. I think we need it, but I am willing to be contradicted. Also: would we be happier defining $\alpha$ only for algebraic integers? There's no reason to, except that we don't consider $D(\alpha)$ in other cases.}
       
       If $\alpha$ is a nonzero algebraic number and $\primeIdealWasl$ is a prime of $\Q(\alpha)$, i.e., a prime ideal in $\mathcal{O}_{\Q(\alpha)}$, then we write $v_{\primeIdealWasl}(\alpha)$ for the valuation of $\alpha$ at $\primeIdealWasl$.

       Our next two propositions show that $D(f(t))$ is finite and that there is a\benw{I removed ``effectively'' because I don't know what it means. Put it back if you insist.} computable upper bound. The explicit formula will be leveraged for the calculations in the next section.

       \begin{proposition}\label{prop:SchinzelDegreeCalculationEasy}
         Let $\alpha$ be a nonzero algebraic number. The positive integer $D(\alpha)$ satisfies the divisibility relation:
         \[        D(\alpha) \mid \gcd\limits_{\primeIdealWasl} (| v_\primeIdealWasl(\alpha)|), \]
         where $\primeIdealWasl$ runs over all prime ideals of $\mathcal{O}_{\Q(\alpha)}$.
       \end{proposition}
       \begin{proof}
         Suppose $\alpha = \theta^D$ for some $\theta \in \Q(\alpha)$ where $D$ is a positive integer. Then $| v_{\primeIdealWasl}(\alpha)| = D | v_{\primeIdealWasl}(\alpha)|$. In particular, $D$ divides $ \gcd_{\primeIdealWasl} (| v_\primeIdealWasl(\alpha)|)$. Taking the supremum with respect to the divisibility ordering yields the result.
       \end{proof}
       If $\alpha$ is an algebraic unit then all the valuations considered in Proposition \ref{prop:SchinzelDegreeCalculationEasy} are $0$, so no useful bound is produced. This is particularly significant for us, since our interest is in integer polynomials whose leading and constant coefficients lie in $\{-1,1\}$, and whose roots therefore are algebraic units. We exploit the structure of $\mathcal{O}_{\Q(\alpha}^\times$ to obtain a bound in this case.

       We adopt the terminology of \cite{Neukirch1999}. In particular, Dirichlet's unit theorem as presented in \cite[Theorem 7.4]{Neukirch1999} and the subsequent discussion says that $\mathcal O_{\Q(\alpha)}(\alpha)$ is a finitely generated abelian group 
       generated by $\zeta_q$, a primitive $q$-th root of unity for some $q$, which generates the torsion subgroup, along with a basis $\{\epsilon_1, \dots, \epsilon_m\}$ for a finitely generated free abelian complementary subgroup. The elements $\epsilon_i \in \Q(\alpha)$ will be called \emph{fundamental units}.

       The choice of fundamental units and $\zeta_q$ yields an isomorphism
       \[ \mathcal O^\times_{\Q(\alpha)} \cong \frac{\Z}{q\Z} \oplus \Z^m;\quad  \zeta_q^{b}\epsilon_1^{c_1}\cdots\epsilon_m^{c_m} \mapsto (b, c_1, \dots, c_m).\]
       	\benw[inline]{I reworked this part pretty heavily to make it easier for me to understand. I included a reference to \cite{Neukirch1999} for the definition of ``fundamental unit''.}

	\begin{proposition} \label{prop:SchinzelDegreeCalculationUnit}
	Suppose $\alpha$ is an algebraic unit that is not a root of unity. Write
	\[\alpha = \zeta_q^{b}\epsilon_1^{c_1}\cdots\epsilon_m^{c_m}.\]

        Let $c = \gcd_i\{c_i\}$. Then
        \begin{equation}
          \label{eq:2}
          D(\alpha) = \prod_{p \textup{ prime }} \begin{cases}
            p^{v_p(c)} & \gcd(p,q) = 1 \\
            \gcd(q, b, p^{v_p(c)}) & \gcd(p,q) \neq 1.
          \end{cases}
        \end{equation}
      \end{proposition}
      Note that only prime numbers $p$ dividing $c$ contribute to the product on the right hand side. Therefore the product is actually finite.
      \begin{proof}
        We remark that $\gcd(q, b, p^{v_p(c)})$ is invariant under replacing $b$ by $b'$ if $b \equiv b' \pmod q$. The number $\alpha$ is also unchanged by this replacement. Therefore we assume, without loss of generality, that $v_p(b) \le v_p(q)$. We write $[b]$ for the reduction of $b$ to $\Z/q\Z$.
        
        Any root of $\alpha$ is also an algebraic unit. Therefore $D(\alpha)$, in this case, is the supremum (under divisibility) of the nonnegative integers $D$ for which $\alpha$ is $D$-divisible in the abelian group $\mathcal O_{\Q(\alpha)}^\times$. We transpose the problem to $\Z/q\Z \oplus \Z^m$, so that we can write our abelian group additively.

        If $D, D'$ are relatively prime, then the Chinese remainder theorem says that an element in an abelian group is $DD'$-divisible if and only if it is both $D$ and $D'$ divisible. Therefore we can determine $D(\alpha)$ by determining, for each prime $p$, the maximal integer $r$ for which $([b], c_1, \dots, c_m) \in \Z/q\Z \oplus \Z^m$ is a multiple of $p^r$. In fact, this $r$ is exactly $v_p(D(\alpha))$. The claim of the proposition is that $r = v_p(c)$ if $p$ does not divide $q$ and $r = \min \{v_p(c), v_p(q), v_p(b)\}$ if it does. Our assumption that $v_p(b) \le v_p(q)$ allows us to simplify this to $r = \min \{v_p(c), v_p(b)\}$.

        The tuple $([b], c_1, \dots, c_m)$ is a multiple of $p^r$ if and only if $[b]$ and $(c_1, \dots, c_m) \in \Z^m$ are multiples of $p^r$. We address each of these separately.

        In the first place, $(c_1, \dots, c_m)$ is a multiple of $p^r$ if and only if $p^r \mid c = \gcd(c_1, \dots c_m)$. In the second, $[b] \in \Z/q\Z$ is a multiple of $p^r$ if and only if one of the following two conditions is satisfied:
        \begin{itemize}
        \item $p$ does not divide $q$, so that $p^r$ is a unit in $\Z/q\Z$;
        \item the lift of $[b]$ to $b \in \Z$ is a multiple of $\gcd(p^r, q)$. We are assuming that $v_p(b) \le v_p(q)$, so that the constraint is exactly $r \le v_p(b)$.
        \end{itemize}

        Putting the above together, we obtain a formula for $v_p(D(\alpha))$, i.e., the largest $r$ for which $(b, c_1, \dots, c_m)$ is a multiple of $p^r$. If $p \nmid q$, then
        \[ v_p(D(\alpha)) = v_p(c).\]
        If $p \mid q$, then $v_p(D(\alpha))$ is the lesser of $v_p(c)$ and $v_p(b)$, i.e.,
        \[v_p(D(\alpha)) = \min\{v_p(c), v_p(b)\},\]
        which is what we wanted to show.          
	\end{proof}

        We remark that modern software packages are capable of producing a system of fundamental units, and so the computation entailed Proposition \ref{prop:SchinzelDegreeCalculationUnit} can be carried out in software.

        The following proposition is a major step towards proving Theorem \ref{thm:main}, and we will also need it in establishing our algorithm in Section~\ref{sec:algorithm}. 
\begin{proposition}\label{prop:Dforirredpowers}
  Let $h(t) \in \Z[t]$ be a primitive irreducible non-cyclotomic polynomial of at least $2$ terms, and let $s$ be a positive integer.

  The integer $D(h(t)^s)$ is positive, and if $h(t)^s$ is $n$-Hartley, then $n \mid D(h(t)^s)$.
\end{proposition}
\begin{proof}
  The roots of the polynomial $h(t)$ are nonzero algebraic numbers $\alpha$ that are not roots of unity. Since $h(t)$ is irreducible, these roots are simple, and for each root $\alpha$ there is an isomorphism $\Q(\alpha) \to \Q[t]/(h(t))$ taking $\alpha$ to $t$. Notably, the fields $\Q(\alpha^s)$ are all isomorphic, so that the integers $D(\alpha^s)$ are all equal as $\alpha$ ranges over the roots of $h(t)$. The $\gcd$ taken over all roots of $D(\alpha^s)$ is exactly $D(h(t)^s)$, so we deduce $D(h(t)^s) = D(\alpha^s)$ for any single root $\alpha$ of $h(t)$.

 If $\alpha$ (and therefore $\alpha^s$) is not an algebraic unit, then Proposition \ref{prop:SchinzelDegreeCalculationEasy} implies $D(\alpha^s)$ is positive, whereas if $\alpha$ is an algebraic unit, Proposition \ref{prop:SchinzelDegreeCalculationUnit} implies this. In particular $D(h(t)^s)$ is a positive integer.

 Corollary \ref{cor:powerRuleFornHartley} says that if $h(t)^s$ is $n$-Hartley, then $n \mid D(\alpha^s)$, which completes the proof.
\end{proof}

We can now give a proof of Theorem \ref{thm:main}.
\begin{proof}[Proof of Theorem \ref{thm:main}]
  Since $\Delta(t)$ is not a product of cyclotomic polynomials, it has an irreducible factor $h(t)$ which is not cyclotomic and appears with some multiplicity $s$. Proposition \ref{prop:Dforirredpowers} implies that $h(t)^s$ is $n$-Hartley only for finitely many values of $n$, and then Proposition \ref{prop:coprimeHartley} implies the same for $\Delta(t)$.
\end{proof}


\section{An algorithm} \label{sec:algorithm}
In this section we describe explicitly an algorithm to determine the complete set of integers $n$ for which a given polynomial $\Delta(t) \in \Z[t]$ is $n$-Hartley. We first note that for any fixed $n$, the existence of a factorization as in Theorem \ref{thm:Hartley} can be checked directly.
\begin{enumerate}
\item If the leading and constant coefficients of $\Delta(t)$ are $a$ and $b$ respectively, then Lemma \ref{lem:leadingConst} tells us that $\Delta(t)$ can be $n$-Hartley only when $|a|$ and $|b|$ are $n$-th powers.
  All such $n$ may be checked directly. Therefore we can assume that the leading and constant coefficients of $\Delta(t)$ lie in $\{-1,1\}$.
  \item By Proposition \ref{prop:coprimeHartley}, we may reduce to the case when $\Delta(t) = h(t)^s$ for an irreducible polynomial $h(t) \in \Z[t]$ and $s \in \N$. Note that the leading and constant coefficients of $h(t)$ are also $\pm 1$.
  \item If $h(t)$ is cyclotomic, then a complete answer is given in Proposition \ref{prop:cycloHartley}.
  \item If $h(t)$ is not cyclotomic, then we compute $D(h(t)^k)$, which is a finite positive integer, and by Proposition \ref{prop:Dforirredpowers} we deduce that $n$ divides the positive integer $D(h(t)^k)$, so that we can check each divisor of $D(h(t)^k)$ directly. By Proposition \ref{prop:primepowers}, we need not check multiples of divisors that have already been ruled out. 
\end{enumerate}
We include a code snippet (in SageMath \cite{sagemath} where the \texttt{pari} functions are wrappers around PARI \cite{PARI2}) for computing the quantity $D(h(t)^k)$. Here $D(h(t),k)$ returns $D(h(t)^k)$, where $h(t)$ is assumed to be irreducible, monic and to have constant coefficient $\pm 1$.

\begin{minipage}{\textwidth}
\begin{python}
def D(h,k):
    field = NumberField(h, "z")
    z = field.gen()
    pari_field = pari(field).bnfinit()
    *free, cyclic = pari_field.bnfisunit(pari(z^k))
    e = gcd(ZZ(elt) for elt in free)
    m = ZZ(pari_field.bnf_get_tu()[0])
    ans = ZZ(1)
    if ZZ(cyclic) == 0:
        ans = e
    else:
        for prime, mult in e.factor():
            if gcd(ZZ(prime), ZZ(m)) == 1:
                ans *= ZZ(prime ** mult)
            else:
                ans *= gcd([ZZ(cyclic), ZZ(m), ZZ(prime ** mult)])
    return ans
\end{python}
\end{minipage}

\section{Applications and examples} \label{sec:computations}
Using the algorithm in Section \ref{sec:algorithm}, we verified Conjecture \ref{conj:algebra} for polynomials of degree at most $36$, which implies that Conjecture \ref{conj:polynomial} is true for all freely periodic knots $K$ with $g(K) \leq 18$. We also verified that none of these polynomials satisfy Murasugi's condition (as in Theorem \ref{thm:Murasugi}), with five excpetions; see Example \ref{exmp:deg30}. We show below that assuming the L-space conjecture implies that these five examples are not counterexamples to Conjecture \ref{conj:polynomial}. Thus we have proved the following theorem. 

	\begin{theorem}
    Conjecture \ref{conj:algebra} is true for all polynomials with degree at most $36$. Conjecture \ref{conj:polynomial} is true for all knots $K$ with genus $g(K) < 13$. If Conjecture \ref{conj:quotient} is true, then Conjecture \ref{conj:polynomial} is true for all knots $K$ with $g(K) < 19$.
	\end{theorem}

    \begin{remark}
    The Alexander polynomials of L-space knots are more restricted than the polynomials in Conjecture \ref{conj:algebra}; see e.g. \cite[Corollary 9]{MR3782416}. These extra restrictions do not appear to be necessary to rule out Hartley factorizations, however.
    \end{remark}
	In order to handle the five exceptional polynomials, we use the following results.
	\begin{proposition} \label{prop:periodic-orderable}
		Let $K$ be an $n$-periodic knot with quotient knot $\overline{K}$ and suppose that $\pi_1(S^3_{p/q}(\overline{K}))$ is left-orderable. Then $\pi_1(S^3_{np/q}(K))$ is also left-orderable.
	\end{proposition}
	
	\begin{proof}
		Let $K$ be an $n$-periodic with quotient knot $\overline{K}$, and suppose that $\pi_1(S^3_{p/q}(\overline{K}))$ is left-orderable. The periodic symmetry on $K$ extends to $S^3_{np/q}(K)$, and the quotient manifold is $S^3_{p/q}(\overline{K})$ (see for example \cite[Section 2.1]{MR1755823}), so that there is a non-trivial homomorphism $\pi_1(S^3_{np/q}(K)) \to \pi_1(S^3_{p/q}(\overline{K}))$. Then by \cite[Theorem 3.2]{MR2141698}, $\pi_1(S^3_{np/q}(K))$ is left-orderable since it has a non-trivial homomorphism into a left-orderable group.
	\end{proof}
	
	\begin{corollary}\label{cor:Lspace}
		The L-space conjecture implies Conjecture \ref{conj:quotient}.
	\end{corollary}
	\begin{proof}
		Let $K$ be an $n$-periodic L-space knot, and suppose that the quotient knot $\overline{K}$ is not an L-space knot. Then there is a surgery $S^3_{p/q}(K)$ which is an L-space, and the corresponding surgery $S^3_{p/(nq)}(\overline{K})$ is not an L-space. By the L-space conjecture \cite[Conjecture 3]{MR3072799}, $\pi_1(S^3_{p/q}(K))$ is not left-orderable, while $\pi_1(S^3_{p/(nq)}(\overline{K}))$ is left orderable, which contradicts Proposition \ref{prop:periodic-orderable}.
	\end{proof}
	
	\begin{example} \label{exmp:deg30}
		Consider the polynomial 
		\[
		\Delta(t) = t^{30} - t^{29} + t^{28} - t^{26} + t^{25} - t^{24} + t^{18} - t^{17} + t^{15} - t^{13} + t^{12} - t^6 + t^5 - t^4 + t^2 - t + 1,
		\]
		which is not a product of cyclotomic polynomials, but appears to be a possible Alexander polynomial of an L-space knot (including having second term non-zero; see \cite{MR3782416}). We do not know if $\Delta(t)$ is the Alexander polynomial of an L-space knot, but it does satisfy Murasugi's condition with $n = 2$ and axis-linking number $3$ two different ways: with factor 
		\[
		\Delta_{\overline{K}}(t) =  t^{14} + t^{12} - t^8 - t^7 - t^6 + t^2 + 1,
		\]
    and with factor
    \[
    \Delta_{\overline{K}}(t) = t^{14} - 2t^{13} + 3t^{12} - 2t^{11} + 2t^9 - 3t^8 + 3t^7 - 3t^6 + 2t^5 - 2t^3 + 3t^2 - 2t + 1.
    \]
		Note that neither possible quotient $\Delta_{\overline{K}}(t)$ can be the Alexander polynomial of an L-space knot since in the first case the coefficients are not alternating, and in the second case the coefficients are not in $\{-1,0,1\}$; see \cite[Corollary 1.3]{MR2168576}. In particular, if $\Delta(t)$ is the Alexander polynomial of a periodic L-space knot which has a quotient with Alexander polynomial $\Delta_{\overline{K}}(t)$, then Conjectures \ref{conj:iteratedtorus}, \ref{conj:polynomial}, and \ref{conj:quotient} (and hence the L-space conjecture) are all false. The polynomials
		\begin{align*} 
		&{\scriptstyle \Delta(t) \, =\,  t^{26} - t^{25} + t^{24} - t^{23} + t^{20} - t^{19} + t^{14} - t^{13} + t^{12} - t^{7} + t^{6} - t^{3} + t^{2} - t + 1,} \\
  &{\scriptstyle \Delta(t) \, =\, t^{36} - t^{35} + t^{34} - t^{33} + t^{27} - t^{26} + t^{25} - t^{23} + t^{20} - t^{19} + t^{18} - t^{17} + t^{16} - t^{13} + t^{11} - t^{10} + t^9 - t^3 + t^2 - t + 1,}\\
   &{\scriptstyle \Delta(t) \, =\, t^{36} - t^{35} + t^{34} - t^{33} + t^{30} - t^{29} + t^{26} - t^{25} + t^{20} - t^{19} + t^{18} - t^{17} + t^{16} - t^{11} + t^{10} - t^7 + t^6 - t^3 + t^2 - t + 1, \textup{ and }}\\
  &{\scriptstyle \Delta(t) \, =\, t^{36} - t^{35} + t^{34} - t^{33} + t^{30} - t^{29} + t^{27} - t^{23} + t^{21} - t^{20} + t^{18} - t^{16} + t^{15} - t^{13} + t^9 - t^7 + t^6 - t^3 + t^2 - t + 1}\\
    \end{align*}
	are similar, in that they each satisfy Murasugi's condition but none of the possible quotients satisfy the conditions from \cite[Corollary 1.3]{MR2168576}. These are the only five such examples through degree $36$.
	\end{example}
  We conclude with an example of an infinite family of L-space knots. 
	\begin{example}
	A standard example of a family of L-space knots is the pretzel knots $$P(-2,3,2q+1).$$ In \cite[Section 2]{MR891592}, it is shown that for $q > 2$, these knots have only a strong inversion and no other symmetries. In particular, this provides some evidence for Conjecture \ref{conj:iteratedtorus}, since none of these knots are periodic or freely periodic. Note that $P(-2,3,3) = T(3,4)$ and $P(-2,3,5) = T(3,5)$ are torus knots; hence the restriction $q > 2$.
	\end{example}

	\printbibliography
\end{document}

